\documentclass[11pt]{amsart}
\usepackage{amsmath,amssymb,latexsym,soul,cite}
\usepackage{color,enumitem,graphicx}
\usepackage[colorlinks=true,urlcolor=blue,
citecolor=red,linkcolor=blue,linktocpage,pdfpagelabels,
bookmarksnumbered,bookmarksopen]{hyperref}
\usepackage[english]{babel}
\usepackage[T1]{fontenc}
\usepackage[hyperpageref]{backref}

\usepackage[left=2.8cm,right=2.8cm,top=2.9cm,bottom=2.9cm]{geometry}

\DeclareMathOperator{\arcosh}{\mathrm{arcosh}}

\numberwithin{equation}{section}

\newtheorem{thm}{Theorem}[section]

  \theoremstyle{plain}
  \newtheorem{lem}[thm]{Lemma}
  \theoremstyle{plain}
  \newtheorem{prop}[thm]{Proposition}
  \theoremstyle{plain}
  
  \theoremstyle{plain}
  
    \theoremstyle{definition}
\newtheorem{rem}[thm]{Remark}

\newcommand{\real}{{\mathbb R}}

\title[On the Moser-Trudinger inequality in Sobolev-Slobodeckij spaces]{On the Moser-Trudinger inequality in fractional Sobolev-Slobodeckij spaces}

\author[E. Parini]{Enea Parini}
\author[B. Ruf]{Bernhard Ruf}

\address[E. Parini]{Aix-Marseille Universit\'e, CNRS, Centrale Marseille, I2M, UMR 7373, 39 Rue Frederic Joliot Curie, 13453 Marseille, France}
\email{enea.parini@univ-amu.fr}
\address[B. Ruf]{Dipartimento di Matematica "Federigo Enriques", Universit\`{a} degli Studi di Milano, via Saldini 50, 20133 Milano, Italy.} \email{bernhard.ruf@unimi.it}

\thanks{}
\subjclass[2010]{}

\thanks{}

\begin{document}

\begin{abstract}
We consider the problem of finding the optimal exponent in the Moser-Trudinger inequality
\[ \sup \left\{\int_\Omega \exp{\left(\alpha\,|u|^{\frac{N}{N-s}}\right)}\,\bigg|\,u \in \widetilde{W}^{s,p}_0(\Omega),\,[u]_{W^{s,p}(\real^N)}\leq 1 \right\}< + \infty.\]
Here $\Omega$ is a bounded domain of $\real^N$ ($N\geq 2$), $s \in (0,1)$, $sp = N$, $\widetilde{W}^{s,p}_0(\Omega)$ is a Sobolev-Slobodeckij space, and $[\cdot]_{W^{s,p}(\real^N)}$ is the associated Gagliardo seminorm. We exhibit an explicit exponent $\alpha^*_{s,N}>0$, which does not depend on $\Omega$, such that the Moser-Trudinger inequality does not hold true for $\alpha \in (\alpha^*_{s,N},+\infty)$.
\end{abstract}

\maketitle

\section{Introduction}

Let $\Omega \subset \real^N$ be a bounded, open domain with Lipschitz boundary. It is well known that
\[ W^{1,N}_0(\Omega) \hookrightarrow L^p(\Omega) \]
for every $p \in [1,+\infty)$, but
\[ W^{1,N}_0(\Omega) \not\hookrightarrow L^\infty(\Omega). \]
A counterexample is given by the function $u(x)=(-\ln{|\ln {|x|} |})_+$, when $\Omega$ is the unit ball. A celebrated result by Trudinger \cite{trudinger} and Moser \cite{moser} states that functions in $W^{1,N}_0(\Omega)$ enjoy summability of exponential type: more precisely, Trudinger proved that there exists $\alpha > 0$ such that
\begin{equation} \label{MT} \sup \left\{\int_\Omega \exp{\left(\alpha\,|u|^{\frac{N}{N-1}}\right)}\,\bigg|\,u \in W^{1,N}_0(\Omega),\,\|\nabla u\|_{L^N}\leq 1 \right\}< + \infty .\end{equation}
Several years later, Moser was able to simplify Trudinger's proof, and to determine the optimal exponent $\alpha_N$ such that \eqref{MT} holds for every $\alpha \in [0,\alpha_N]$, and fails for $\alpha \in (\alpha_N, + \infty)$. This optimal exponent is given by $\alpha_N = N(N\omega_N)^{\frac{1}{N-1}}$, where $\omega_N$ is the volume of the $N$-dimensional unit ball. In particular, for $N=2$, $\alpha_2=4\pi$. While the proof of the validity of \eqref{MT} for $\alpha \in [0,\alpha_N)$ is not difficult, the very delicate point is to prove that it holds also for $\alpha=\alpha_N$. This is done by showing, after reducing by symmetrization to the case where $\Omega$ is the unit ball, that if $\{v_k\}$ is a maximizing sequence, it can not be ``too far'' from the so-called Moser sequence given by
\begin{equation} \label{maxsequence0} u_k(x)=\left\{\begin{array}{c l} |\ln k|^{\frac{N-1}{N}} & \text{if }|x| \leq \frac{1}{k} \\ \displaystyle \frac{|\ln{|x|}|}{|\ln{k}|^{\frac{1}{N}}} & \text{if }\frac{1}{k} < |x| < 1 \\ 0 & \text{if } |x|= 1 \end{array}\right.\end{equation}
The same sequence is also used to prove the failure of \eqref{MT} for $\alpha \in (\alpha_N, +\infty)$. Subsequently, Adams \cite{adams} was able to extend the results to higher order Sobolev spaces $W^{k,p}_0(\Omega)$ with $kp = N$. His proof is based on expressing functions belonging to the space as Riesz potentials of their gradients of order $k$. This approach can be extended to Sobolev spaces of fractional orders. In \cite[Theorem 3.1 (ii)]{xiaozhai}, the authors prove a fractional Moser-Trudinger inequality of the kind
\[ \int_\Omega \text{exp}\left(\alpha \, |u|^{\frac{N}{N-s}}\right) \leq F_{p,N}\,|\Omega| \]
for every function such that $(-\Delta)^{\frac{s}{2}} u \subset \subset \Omega$ and $(-\Delta)^{\frac{s}{2}} u \in L^p(\Omega)$ with $\|(-\Delta)^{\frac{s}{2}} u\|_p \leq 1$. However, due to the nonlocality of the fractional Laplacian, this is not the same as requiring that the function has compact support in $\Omega$. A result in this direction has been given recently in \cite{martinazzi}. However, the case of the Sobolev-Slobodeckij spaces $W^{s,p}_0(\Omega)$ or ${\widetilde W}^{s,p}_0(\Omega)$, whose definitions are given in the following section, has received considerably less attention. The existence of a $\alpha_* > 0$ such that the corresponding version of \eqref{MT} is satisfied for $\alpha \in (0,\alpha_*)$ is essentially proved in \cite{peetre}, as we point out in Theorem \ref{validity}. However, the value of the optimal exponent is not known. In this paper, using a slightly modified version of the sequence \eqref{maxsequence0}, we give an explicit exponent $\alpha^*_{s,N}$ such that the Moser-
Trudinger inequality does not hold true for $\alpha \in (\alpha^*_{s,N},
+\infty)$. Here the precise 
statement of the main result:
\begin{thm} 
Let $\Omega$ be a bounded, open domain of $\real^N$ ($N\geq 2$) with Lipschitz boundary, and let $s \in (0,1)$, $sp = N$. Let $\widetilde{W}^{s,p}_0(\Omega)$ be the space defined as the completion of $C^\infty_c(\Omega)$ with respect to the norm
\[ u \mapsto \left( \|u\|_{L^p(\Omega)}^p + [u]_{W^{s,p}(\real^N)}^p \right)^\frac{1}{p},\]
where
\[ [u]_{W^{s,p}(\real^N)} := \int_{\real^N} \int_{\real^N} \frac{|u(x)-u(y)|^p}{|x-y|^{N+sp}}\,dx\,dy. \]
Then there exists $\alpha_* = \alpha_*(s,\Omega)> 0$ such that
\[ \sup \left\{\int_\Omega \exp{\left(\alpha\,|u|^{\frac{N}{N-s}}\right)}\,\bigg|\,u \in \widetilde{W}^{s,p}_0(\Omega),\,[u]_{W^{s,p}(\real^N)}\leq 1 \right\}< + \infty \qquad \text{for } \alpha \in [0,\alpha_*).\]
Moreover,
\[ \sup \left\{\int_\Omega \exp{\left(\alpha\,|u|^{\frac{N}{N-s}}\right)}\,\bigg|\,u \in \widetilde{W}^{s,p}_0(\Omega),\,[u]_{W^{s,p}(\real^N)}\leq 1 \right\} = + \infty \qquad \text{for } \alpha \in (\alpha^*_{s,N}, +\infty),\] where
\[ \alpha^*_{s,N} := N\left(\frac{2\,(N\omega_N)^2\,\Gamma(p+1)}{N!}\,\sum_{k=0}^\infty \frac{(N+k-1)!}{k!}\frac{1}{(N+2k)^p}\right)^\frac{s}{N-s}.\]
\end{thm}
The paper is structured as follows. After stating some preliminary results about fractional Sobolev spaces and special functions, in Section 3 we prove the validity of the Moser-Trudinger inequality for some exponent $\alpha > 0$. In Section 4 we establish a formula for the Gagliardo seminorm of a radially symmetric function, which will be needed in Section 5 to prove the failure of the Moser-Trudinger inequality when the exponent is too big. We conclude with some final remarks and open questions.

The authors would like to thank Lorenzo Brasco for some useful and interesting discussions.

\section{Notation and preliminary results}

The $N$-dimensional volume of a unit ball in $\real^N$ will be denoted by $\omega_N$. It is then well known that the $(N-1)$-dimensional surface measure $S_{N-1}$ of a $N$-sphere is equal to $N\omega_N$. $\omega_N$ has the explicit expression
\[ \omega_N = \frac{\pi^{\frac{N}{2}}}{\Gamma\left(\frac{N}{2}+1\right)}.\]

\subsection{Sobolev-Slobodeckij spaces}

Let $N \geq 2$, and let $\Omega \subset \real^N$ be a bounded, open set with Lipschitz boundary. For $s\in (0,1)$, $p\in [1,+\infty)$, we define the quantities
\[ [u]_{W^{s,p}(\Omega)} := \left( \int_\Omega \int_\Omega \frac{|u(x)-u(y)|^p}{|x-y|^{N+sp}}\,dx\,dy \right)^{\frac{1}{p}}, \]
\[ [u]_{W^{s,p}(\real^N)} := \left( \int_{\real^N} \int_{\real^N} \frac{|u(x)-u(y)|^p}{|x-y|^{N+sp}}\,dx\,dy \right)^{\frac{1}{p}}. \]
The \emph{Sobolev-Slobodeckij space} $W^{s,p}(\Omega)$ is defined as
\[ W^{s,p}(\Omega) := \left\{ u \in L^p(\Omega)\,|\, [u]_{W^{s,p}(\Omega)}< + \infty \right\}\]
which is a Banach space when endowed with the norm
\[ \|u\|_{W^{s,p}(\Omega)} := \left(\|u\|_{L^p(\Omega)}^p + [u]_{W^{s,p}(\Omega)}^p\right)^{\frac{1}{p}}. \]
Moreover, let us define the spaces
\[ W^{s,p}_0(\Omega) = \overline{C^\infty_0(\Omega)}^{\|\cdot\|_{W^{s,p}(\Omega)}},\]
\[ \widetilde{W}^{s,p}_0(\Omega) = \overline{C^\infty_0(\Omega)}^{\|\cdot\|_{W^{s,p}(\real^N)}}.\]
The space $\widetilde W^{s,p}_0(\Omega)$ can be equivalently defined by taking the completion of $C^\infty_0(\Omega)$ with respect to the seminorm $[u]_{W^{s,p}(\real^N)}$ (see for example \cite[Remark 2.5]{brascolindgrenparini}). If $\partial\Omega$ is Lipschitz, it holds
\[ \widetilde{W}^{s,p}_0(\Omega)= \left\{u\in L^p(\real^N)\, :\,u \equiv 0 \text{ in }\real^N \setminus \Omega,\, \|u\|_{W^{s,p}(\real^N)}<+\infty\right\},\]
(see \cite[Proposition B.1]{brascoparinisquassina}). If moreover $s\,p\not =1$, then
\[ W^{s,p}_0 (\Omega) = \widetilde{W}^{s,p}_0(\Omega) \]
(see \cite[Proposition B.1]{brascolindgrenparini}). In the limit $s \to 1^-$, one recovers the usual Sobolev space $W^{1,p}_0(\Omega)$. In particular, for $u \in W^{1,p}_0(\Omega)$,
\begin{equation} \label{essea1} \lim_{s \to 1^-} (1-s)[u]_{W^{s,p}(\Omega)}^p = \lim_{s \to 1^-} (1-s)[u]_{W^{s,p}(\real^N)}^p = K(p,N)\|\nabla u\|_{L^p(\Omega)}^p,\end{equation}
where
\[ K(p,N) = \frac{1}{p}\int_{S^{N-1}} |\langle \sigma,\mathbf{e}\rangle|^p\,d\mathcal{H}^{N-1}(\sigma)\]
(see \cite{bourgainbrezismironescu} and \cite[Proposition 2.8]{brascoparinisquassina}). The Sobolev-Slobodeckij spaces can also be defined equivalently by means of interpolation theory, as in \cite{peetre}. It is important to observe that the Sobolev-Slobodeckij spaces are in general different from the Bessel potential spaces $H^{s,p}$. Indeed, if $\Omega = \real^N$, the latter is defined as
\[ H^{s,p}(\real^N) := \left\{ u \in L^1_{loc}(\real^N)\,|\,(I-\Delta )^{\frac{s}{2}}u \in L^p(\real^N) \right\}\]
or, in terms of Fourier transform,
\[ H^{s,p}(\real^N) := \left\{ u \in L^1_{loc}(\real^N)\,|\,\mathcal{F}^{-1} (1+|\xi|^2)^{\frac{s}{2}} \mathcal{F}u \in L^p(\real^N) \right\}\]
(see \cite[Theorem 3.7]{kurokawa}). This space coincides with the Triebel-Lizorkin space $F^s_{p,2}(\real^N)$ \cite[Section 2.5.6]{triebel}. On the other hand, $W^{s,p}(\real^N) = B^s_{p,p}(\real^N)$, where $B^s_{p,p}(\real^N)$ is a Besov space, and $B^s_{p,p}(\real^N) = F^s_{p,p}(\real^N)$ \cite[Section 2.3.9]{triebel}. However, $F^s_{p,p}(\real^N) \neq F^s_{p,2}(\real^N)$ if $p \neq 2$ \cite[Section 2.5.6]{triebel}.

\subsection{Special functions}

In the following we will recall the definitions and the properties of some special functions which will be needed throughout the paper.

\subsubsection{Gamma function} The Gamma function is defined on $\mathbb{C}\setminus \{0,-1,-2...\}$, and for $Re(z)>0$ it has the integral representation
\[ \Gamma(z) := \int_0^{+\infty} t^{z-1}e^{-t}\,dt.\]
It is well known that $\Gamma(z+1)=z\Gamma(z)$ for every $z$ in the domain of definition, and $\Gamma(n+1)=n!$ for $n \in \mathbb{N}$. 

\subsubsection{Beta function} The Beta function is defined on $\{ (x,y)\in \mathbb{C}\times \mathbb{C}\,|\, Re(x)>0,\,Re(y)>0\}$. It has the integral representation
\[ B(x,y) := \int_0^1 t^{x-1}(1-t)^{y-1}\,dt.\]
It is well known that
\begin{equation} \label{betagamma} B(x,y) = \frac{\Gamma(x)\,\Gamma(y)}{\Gamma(x+y)}. \end{equation}

\subsubsection{Hypergeometric function}

The hypergeometric function is defined for $|z|<1$ by the power series
\[ {_{2} F}_1(a,b;c;z) = \sum_{n=0}^\infty \frac{(a)_n\,(b)_n}{(c)_n}\,\frac{z^n}{n!}, \]
where $(q)_n = q(q+1)...(q+n-1)$ if $n>0$, and $(q)_0=1$. We will need the following result.
\begin{lem} \label{hyper}
Let ${_{2} F}_1$ be the hypergeometric function. Then,
\[ {_{2} F}_1 \left(\frac{N}{2},\frac{N+1}{2};\frac{N}{2};z\right) = (1-z)^{-(N+1)/2}, \]
\end{lem}
\begin{proof}
From the relation
\[{_{2} F}_1(a,b;c,z) = {_{2} F}_1 (c-a,c-b;c,z)\cdot (1-z)^{-(a+b-c)}\]
\cite[Proposition 15.3.3]{abramowitzstegun} we obtain
\[{_{2} F}_1 \left(\frac{N}{2},\frac{N+1}{2};\frac{N}{2};z\right) = {_{2} F}_1 \left(0,-\frac{1}{2};\frac{N}{2},z\right)\cdot(1-z)^{-(N+1)/2}.\]
Since $N>1$, by \cite[Proposition 15.3.1]{abramowitzstegun} and relation \eqref{betagamma} we can write
\[{_{2} F}_1 \left(0,-\frac{1}{2};\frac{N}{2},z\right) = \frac{\Gamma\left(\frac{N}{2}\right)}{\Gamma\left(-\frac{1}{2}\right)\Gamma\left(\frac{N+1}{2}\right)}\int_0^1 t^{-\frac{3}{2}}(1-t)^{\frac{N-1}{2}} = \frac{\Gamma\left(\frac{N}{2}\right)}{\Gamma\left(-\frac{1}{2}\right)\Gamma\left(\frac{N+1}{2}\right)} B\left(-\frac{1}{2},\frac{N+1}{2}\right)=1.\]
\end{proof}

\subsubsection{Associated Legendre functions} The associated Legendre function of the second kind is defined as
\[ Q_{\nu}^{\mu}(z)= e^{i\mu\pi}\,\frac{\sqrt{\pi}\,\Gamma(\mu+\nu+1)}{2^{\nu+1}\,\Gamma(\nu+3/2)}\,\frac{(z^2-1)^{\mu/2}}{z^{\mu+\nu+1}} {_{2} F}_1 \left(1+\frac{\mu}{2}+\frac{\nu}{2},\frac{1}{2}+\frac{\mu}{2}+\frac{\nu}{2};\nu+\frac{3}{2};\frac{1}{z^2}\right) \qquad \text{if }|z|>1 \]
(see \cite[8.1.3]{abramowitzstegun}).

\subsubsection{Modified Bessel functions} For $\alpha \in \real$, the modified Bessel function $I_\alpha$ of order $\alpha$ is the solution to the differential equation
\[ z^2\,y''(z)+z\,y'(z)-(z^2+\alpha^2)\,y(z)=0 \]
given by the series development
\[ I_\alpha(z) = \sum_{k=0}^{\infty} \frac{1}{k!\,\Gamma(\alpha + k+1)}\left(\frac{z}{2}\right)^{2k+\alpha}\]
(see \cite[9.6.10]{abramowitzstegun}).

\subsubsection{Zeta function} The Riemann zeta function is the analytic continuation of the function defined, for $Re(s)>1$, by the series
\[ \zeta(s) = \sum_{k=1}^\infty \frac{1}{k^s}. \]
The Riemann zeta function is meromorphic with a single pole at $s=1$, and can be written as the Laurent series
\[ \zeta(s) = \frac{1}{s-1} + \sum_{k=0}^\infty \frac{(-1)^k}{k!}\gamma_k(s-1)^k, \]
where $\gamma_k$ are the Stieltjes constants (see \cite[23.2.5]{abramowitzstegun}). This implies that
\begin{equation} \label{limitezeta} \lim_{s \to 1} (s-1)\zeta(s) = 1.\end{equation}
The Hurwitz zeta function is defined, for $Re(s)>1$ and $Re(q)>0$, as
\[ \zeta(s,q) = \sum_{k=0}^\infty \frac{1}{(q+k)^s}.\]
Clearly, $\zeta(s,1)=\zeta(s)$. Moreover, it holds
\begin{equation} \label{hurwitz}  \zeta\left(s,\frac{1}{2}\right) = (2^s-1)\zeta(s).\end{equation}
(see \cite[p. 41]{ivic}).

\section{Validity of Moser-Trudinger inequality}

In this section we give a short proof of the validity of the Moser-Trudinger inequality for some value of $\alpha >0$, which is essentially contained in \cite{peetre}, following the approach by Trudinger.

\begin{prop} \label{validity}
Let $\Omega$ be a bounded, open domain of $\real^N$, and let $sp = N$. There exists $\alpha_* = \alpha_*(s,\Omega)> 0$ such that
\[ \sup \left\{\int_\Omega \exp{\left(\alpha\,|u|^{\frac{N}{N-s}}\right)}\,\bigg|\,u \in \widetilde{W}^{s,p}_0(\Omega),\,[u]_{W^{s,p}(\real^N)}\leq 1 \right\}< + \infty \qquad \text{for } \alpha \in [0,\alpha_*).\]
\end{prop}
\begin{proof}
By \cite[Theorem 9.1]{peetre} we have
\[ \sup_{q>p} \left(\sup_{u \in \widetilde{W}^{s,p}_0(\Omega)} \frac{\|u\|_{q}}{q^{\frac{N-s}{N}}\,[u]_{W^{s,p}(\real^N)}}\right) < +\infty.\]
This holds true also for $1 \leq q \leq p$, as a consequence of H\"{o}lder's inequality and Sobolev embedding. Therefore there exists a constant $C>0$, depending on $\Omega$, such that for every $u \in \widetilde{W}^{s,p}_0(\Omega)$
\[ \|u\|_{q} \leq C\,[u]_{W^{s,p}(\real^N)}\,q^{\frac{N-s}{N}}.\]
If $[u]_{W^{s,p}(\real^N)} \leq 1$, using the series development of the exponential function we have
\[ \int_\Omega \exp{\left(\alpha |u|^{\frac{N}{N-s}}\right)} = \sum_{k=0}^\infty \frac{\alpha^k}{k!} \int_\Omega |u|^{\frac{kN}{N-s}} \leq \sum_{k=0}^\infty \frac{1}{k!} \left(\frac{CN}{N-s}\alpha k\right)^{k}.\]
The last series is convergent for $\alpha$ small enough thanks to Stirling's formula
\[ k! = \sqrt{2\pi k} \left(\frac{k}{e}\right)^k \left(1+O\left(\frac{1}{k}\right)\right),\] 
and it gives a uniform bound for
\[ \int_\Omega \exp{\left(\alpha |u|^{\frac{N}{N-s}}\right)} \]
not depending on the function $u$.
\end{proof}

\begin{prop}
Let $\Omega$ be a bounded, open domain of $\real^N$, and let $sp = N$. If $u \in \widetilde{W}^{s,p}_0(\Omega)$, for every $\alpha > 0$ it holds
\[ \exp{\left(\alpha\,|u|^{\frac{N}{N-s}}\right)} \in L^1(\Omega). \]
\end{prop}
\begin{proof}
For the sake of simplicity, we will write $[\,\cdot\,] = [\,\cdot\,]_{W^{s,p}(\real^N)}$. By definition of $\widetilde{W}^{s,p}_0(\Omega)$, there exists a sequence $\{u_k\}$ in $C^\infty_c(\Omega)$ such that $[u_k - u] \to 0$ as $k \to \infty$. Then it is possible to write $u := v+w$, where $v \in C^\infty_c(\Omega)$, and $[w] \leq \frac{1}{2}\left(\frac{\alpha_*}{\alpha}\right)^{\frac{N-s}{N}}$. By convexity,
\[ \exp{\left(\alpha\,|u|^{\frac{N}{N-s}}\right)} \leq \exp\left(2^{\frac{s}{N-s}}\alpha\,(|v|^{\frac{N}{N-s}}+|w|^{\frac{N}{N-s}})\right). \]
Clearly,
\[ \exp\left(2^{\frac{s}{N-s}}\alpha\,|v|^{\frac{N}{N-s}}\right) \in L^\infty(\Omega).\]
Moreover,
\begin{align*} \exp\left(2^{\frac{s}{N-s}}\alpha\,|w|^{\frac{N}{N-s}}\right) & = \exp\left(2^{\frac{s}{N-s}}\alpha[w]^{\frac{N}{N-s}}\,\left(\frac{|w|}{[w]}\right)^{\frac{N}{N-s}}\right) \\ & \leq \exp\left(\frac{\alpha_*}{2}\left(\frac{|w|}{[w]}\right)^{\frac{N}{N-s}}\right) \in L^1(\Omega)\end{align*}
by Proposition \ref{validity}.
\end{proof}

\section{Gagliardo seminorm of radially symmetric functions}

In this section we will give a formula for the Gagliardo seminorm of a radially symmetric function $u \in W^{s,p}(\real^N)$, which will be needed in the following, and which might be of independent interest. We will need a couple of technical lemmas.

\begin{lem} \label{sphere}
Let $\beta > 0$. Then
\[ \int_{S^{N-1}} e^{\beta \cos(\mathbf{e},\sigma)}\,d\mathcal{H}^{N-1}(\sigma) = 2\pi^\frac{N}{2}\,\left( \frac{2}{\beta} \right)^{\frac{N}{2}-1}\,I_{\frac{N}{2}-1}(\beta),\]
where $I_{\frac{N}{2}-1}$ is the modified Bessel function of order $\frac{N}{2}-1$.
\end{lem}
\begin{proof}
We can write, according to \cite[Proposition 3.915]{gradshteyn},
\begin{align*} & \int_{S^{N-1}} e^{\beta \cos(\mathbf{e},\sigma)}\,d\mathcal{H}^{N-1}(\sigma) = (N-1)\,\omega_{N-1}\,\int_0^\pi e^{\beta \cos {\theta}}\,(\sin{\theta})^{N-2}\,d\theta \\ & = (N-1)\,\omega_{N-1}\,\sqrt{\pi}\,\left( \frac{2}{\beta} \right)^{\frac{N}{2}-1}\,\Gamma\left(\frac{N-1}{2}\right)\,I_{\frac{N}{2}-1}(\beta) \\ & = 2\pi^\frac{N}{2}\,\left( \frac{2}{\beta} \right)^{\frac{N}{2}-1}\,\,I_{\frac{N}{2}-1}(\beta)\end{align*}
where we used the relations $\omega_{N-1} = \frac{S_{N-2}}{N-1}$ and $S_{N-2}= \frac{2\pi^\frac{N-1}{2}}{\Gamma\left(\frac{N-1}{2}\right)}$.
\end{proof}

\begin{lem} \label{integralbessel2}
Let $r,t \in \real^+$ be such that $r \neq t$. Then,
\[ \int_0^{+\infty} h^{\frac{N}{2}}\,e^{-h(r^2+t^2)}I_{\frac{N}{2}-1}(2rth)\,dh = \frac{\Gamma(N)}{\Gamma(N/2)}\,(rt)^{N/2-1}\,\frac{r^2 + t^2}{|r^2-t^2|^{N+1}}.
\] 
\end{lem}
\begin{proof}
By \cite[Proposition 6.622]{gradshteyn} it holds
\[ \int_0^{+\infty} x^{\frac{N}{2}}\,e^{-x\,\cosh{\alpha}}I_{\frac{N}{2}-1}(x)\,dx = e^{-i\pi(N+1)/2}\,\sqrt{\frac{2}{\pi}}\,\frac{Q_{N/2-3/2}^{N/2+1/2}(\cosh{\alpha})}{(\sinh{\alpha})^{N/2+1/2}},\]
provided $\cosh{\alpha}>1$, where $Q_{N/2-3/2}^{N/2+1/2}$ is an associated Legendre function whose definition is
\[ Q_{N/2-3/2}^{N/2+1/2}(z)= e^{i\pi(N+1)/2}\,\frac{\sqrt{\pi}\,\Gamma(N)}{2^{N/2-1/2}\,\Gamma(N/2)}\,\frac{(z^2-1)^{(N+1)/4}}{z^{N}} {_{2} F}_1 \left(\frac{N}{2},\frac{N+1}{2};\frac{N}{2};\frac{1}{z^2}\right) \qquad \text{if }|z|>1. \]
If we observe that, by Lemma \ref{hyper},
\[ {_{2} F}_1 \left(\frac{N}{2},\frac{N+1}{2};\frac{N}{2};z\right) = (1-z)^{-(N+1)/2}, \]
the above integral simplifies to
\[ \int_0^{+\infty} x^{\frac{N}{2}}\,e^{-x\,\cosh{\alpha}}I_{\frac{N}{2}-1}(x)\,dx = \frac{\Gamma(N)}{2^{N/2-1}\,\Gamma(N/2)}\,\frac{\cosh{\alpha}}{(\cosh^2{\alpha}-1)^{(N+1)/2}}.\]
Setting $\alpha = \arcosh \left(\frac{r^2+t^2}{2rt} \right)$ (which is strictly positive if $r \neq t$) and using the change of variable $x= 2rth$ we obtain
\begin{align*} & \int_0^{+\infty} h^{\frac{N}{2}}\,e^{-h(r^2+t^2)}I_{\frac{N}{2}-1}(2rth)\,dh = \frac{1}{(2rt)^{N/2+1}}\,\int_0^{+\infty} x^{\frac{N}{2}}\,e^{-x\,\cosh{\alpha}}I_{\frac{N}{2}-1}(x)\,dx \\ & = \frac{1}{(2rt)^{N/2+1}}\,\frac{\Gamma(N)}{2^{N/2-1}\,\Gamma(N/2)}\,\frac{\cosh{\alpha}}{(\cosh^2{\alpha}-1)^{(N+1)/2}} = \frac{\Gamma(N)}{\Gamma(N/2)}\,(rt)^{N/2-1}\,\frac{r^2 + t^2}{|r^2-t^2|^{N+1}} .  \end{align*}
\end{proof}

We are now ready to give the proof of the main result of this section.

\begin{prop}
Let $u \in W^{s,p}(\real^N)$ be a radially symmetric function. Suppose that $sp=N$. Then,
\begin{equation} \label{integraleradiale} \int_{\real^N} \int_{\real^N} \frac{|u(x)-u(y)|^p}{|x-y|^{N+s\,p}}\,dx\,dy = (N\omega_N)^2\,\int_0^{+\infty}  \int_0^{+\infty} \,|u(r)-u(t)|^p \,r^{N-1}\,t^{N-1}\, \frac{r^2 + t^2}{|r^2-t^2|^{N+1}}\,dr\,dt.\end{equation}
\end{prop}
\par \medskip
\begin{proof}

By \cite[Section 9.1]{almgrenlieb},
\begin{align*} \int_{\real^N} \int_{\real^N} \frac{|u(x)-u(y)|^p}{|x-y|^{N+s\,p}}\,dx\,dy & = \frac{1}{\Gamma(\frac{N+s\,p}{2})}\int_0^{+\infty} \left( \int_{\real^N} \int_{\real^N} |u(x)-u(y)|^p\,e^{-h\,|x-y|^2}\,dx\,dy\right) h^{\frac{N+s\,p}{2}-1}\,dh \\ & = \frac{1}{\Gamma(N)}\int_0^{+\infty} \left( \int_{\real^N} \int_{\real^N} |u(x)-u(y)|^p\,e^{-h\,|x-y|^2}\,dx\,dy\right) h^{N-1}\,dh.\end{align*}
If $\mathbf{e} \in \real^N$ is a unit vector, we have, thanks to Lemma \ref{sphere},
\begin{align*} & \int_{\real^N} \int_{\real^N} |u(x)-u(y)|^p\,e^{-h\,|x-y|^2}\,dx\,dy  \\ & = N\,\omega_N \int_0^{+\infty} r^{N-1}\,\left( \int_{\real^N} |u(r{\bf e})-u(y)|^p\,e^{-h\,|r{\bf e}-y|^2}\,dy \right)\,dr \\ & = N\,\omega_N \int_0^{+\infty} r^{N-1} \left( \int_{S^{N-1}} \left( \int_0^{+\infty} t^{N-1}\,|u(r)-u(t)|^p\,e^{-h(r^2+t^2-2rt\cos(\mathbf{e},\sigma))}\,dt\right)\,d\mathcal{H}^{N-1}(\sigma) \right)\,dr \\ & = N\,\omega_N \int_0^{+\infty} r^{N-1}  \left( \int_0^{+\infty} t^{N-1}\,|u(r)-u(t)|^p\,e^{-h(r^2+t^2)}\left(\int_{S^{N-1}} e^{2rth\cos(\mathbf{e},\sigma)}\,d\mathcal{H}^{N-1}(\sigma)\right)\,dt\right)\,dr \\ & = 2\pi^\frac{N}{2}N\,\omega_N \int_0^{+\infty} r^{N-1}  \left( \int_0^{+\infty} t^{N-1}\,|u(r)-u(t)|^p\,e^{-h(r^2+t^2)}\,(rth)^{1-\frac{N}{2}}\,I_{\frac{N}{2}-1}(2rth)\,dt\right)\,dr. \end{align*}
\par \noindent
Performing an integration in the variable $h$ and appealing to Lemma \ref{integralbessel2} we obtain
\begin{align*}
 \int_{\real^N} \int_{\real^N} \frac{|u(x)-u(y)|^p}{|x-y|^{N+s\,p}}\,dx\,dy = \frac{2\pi^\frac{N}{2}N\,\omega_N}{\Gamma(N/2)}\,  \int_0^{+\infty} \int_0^{+\infty} |u(r)-u(t)|^p\,r^{N-1}\,t^{N-1}\,\frac{r^2 + t^2}{|r^2-t^2|^{N+1}}\,dt\,dr.
\end{align*}
It now remains to observe that
\[ \frac{2\pi^\frac{N}{2}N\,\omega_N}{\Gamma(N/2)} = \frac{4\,\pi^{N}}{\Gamma(N/2)^2} = (N\omega_N)^2.\]
\end{proof}

\begin{rem}
Let $C(N):=(N\omega_N)^2$ be the constant appearing in \eqref{integraleradiale}. We have $C(2)=4\pi^2$, $C(3)=16\pi^2$, $C(4)=4\pi^4$, $C(5)=\frac{64}{9}\pi^4$.
\end{rem}

\section{Upper bound for the optimal exponent} \label{upper}

In order to give an upper bound to the optimal exponent $\overline{\alpha}$ such that
\begin{equation} \label{MTopt} \sup \left\{\int_\Omega \exp{\left(\alpha\,|u|^{\frac{N}{N-s}}\right)}\,\bigg|\,u \in \widetilde{W}^{s,p}_0(\Omega),\,[u]_{W^{s,p}(\real^N)}\leq 1 \right\}< + \infty \end{equation}
for $\alpha \in [0,\overline{\alpha})$, it is enough to consider the case where $\Omega$ is a ball. Indeed, let $B_r \subset \Omega$ be a ball of radius $r$, and let $\{u_k\}$ be a sequence in $\widetilde{W}^{s,p}_0(B_r)$ such that $[u]_{W^{s,p}(\real^N)} \leq 1$ and 
\[ \int_{B_r} \exp{\left(\alpha\,|u_k|^{\frac{N}{N-s}}\right)} \to + \infty \qquad \text{as }k \to +\infty.\]
For every $k$, let $\widetilde{u}_k \in \widetilde{W}^{s,p}_0(\Omega)$ the function defined by extending $u_k$ to zero outside $B_r$. It holds
\[ \int_\Omega \exp{\left(\alpha\,|\widetilde{u}_k|^{\frac{N}{N-s}}\right)} \geq \int_{B_r} \exp{\left(\alpha\,|u_k|^{\frac{N}{N-s}}\right)} \quad \text{and} \quad  [\widetilde{u}_k]_{W^{s,p}(\real^N)} = [u_k]_{W^{s,p}(\real^N)} \leq 1,\]
and therefore
\[ \int_\Omega \exp{\left(\alpha\,|\widetilde{u}_k|^{\frac{N}{N-s}}\right)} \to + \infty \qquad \text{as }k \to +\infty,\]
so that
\[ \sup \left\{\int_\Omega \exp{\left(\alpha\,|u|^{\frac{N}{N-s}}\right)}\,\bigg|\,u \in \widetilde{W}^{s,p}_0(\Omega),\,[u]_{W^{s,p}(\real^N)}\leq 1 \right\}= + \infty. \]
Moreover, by a simple scaling argument it is easy to see that the optimal exponent does not depend on the radius $r$.

Let us denote by $B$ the unit ball. We consider the family of functions defined by
\begin{equation} \label{maxsequence} u_\varepsilon(x)=\left\{\begin{array}{c l} |\ln \varepsilon|^{\frac{N-s}{N}} & \text{if }|x| \leq \varepsilon \vspace{0.2cm} \\ \displaystyle \frac{\big|\ln{|x|}\big|}{|\ln{\varepsilon}|^{\frac{s}{N}}} & \text{if }\varepsilon < |x| < 1 \vspace{0.2cm} \\ 
0 & \text{if } |x|\geq 1 \end{array}\right.\end{equation}
whose restrictions to $B$ belong to $\widetilde{W}^{s,p}_0(B)$. For $s=1$, this is the Moser-sequence used in \cite{moser}, which satisfies
\[ \|\nabla u_\varepsilon\|_2^2 = N\omega_N \quad \text{for every }\varepsilon > 0.\]
For $s \in (0,1)$, we cannot expect that $[u_\varepsilon]_{W^{s,p}(\real^N)}$ is constant, therefore it is important to compute the limit as $\varepsilon \to 0$ of the quantity 
\begin{align*} I(\varepsilon) & := \int_{\real^N} \int_{\real^N} \frac{|u_\varepsilon(x)-u_\varepsilon(y)|^p}{|x-y|^{N+sp}}\,dx\,dy \\ &= (N\omega_N)^2 \int_0^{+\infty}  \int_0^{+\infty} \,|u_\varepsilon(s)-u_\varepsilon(t)|^p \,r^{N-1}\,t^{N-1}\, \frac{r^2 + t^2}{|r^2-t^2|^{N+1}}\,dr\,dt\end{align*}
which we can decompose into $I(\varepsilon)=I_1(\varepsilon)+I_2(\varepsilon)+I_3(\varepsilon)+I_4(\varepsilon)$, where
\[ I_1(\varepsilon) = \frac{2}{|\ln{\varepsilon}|}\int_\varepsilon^1 \int_0^\varepsilon |\ln{r}-\ln{\varepsilon}|^p\,r^{N-1}\,t^{N-1}\, \frac{r^2 + t^2}{|r^2-t^2|^{N+1}}\,dr\,dt,\]
\[ I_2(\varepsilon) = \frac{1}{|\ln{\varepsilon}|}\int_\varepsilon^1 \int_\varepsilon^1 |\ln{r}-\ln{t}|^p\,r^{N-1}\,t^{N-1}\, \frac{r^2 + t^2}{|r^2-t^2|^{N+1}}\,dr\,dt, \]
\[ I_3(\varepsilon) =  2|\ln{\varepsilon}|^{p-1}\int_1^{+\infty} \int_0^\varepsilon r^{N-1}\,t^{N-1}\, \frac{r^2 + t^2}{|r^2-t^2|^{N+1}}\,dr\,dt, \]
\[ I_4(\varepsilon) = \frac{2}{|\ln{\varepsilon}|}\int_\varepsilon^1 \int_1^{+\infty} |\ln{r}|^p\,r^{N-1}\,t^{N-1}\, \frac{r^2 + t^2}{|r^2-t^2|^{N+1}}\,dr\,dt. \]
\par \smallskip \noindent
The following basic calculus fact will be extremely useful for the calculations:
\[ \frac{d}{dr}\left(\frac{1}{N} \frac{r^N}{(t^2-r^2)^N}\right) = \frac{r^{N-1}(t^2+r^2)}{(t^2-r^2)^{N+1}}. \]

\subsection{Computations for $I_1(\varepsilon)$}

\begin{align*}
 & \frac{2}{|\ln{\varepsilon}|}\int_\varepsilon^1 \int_0^\varepsilon |\ln{r}-\ln{\varepsilon}|^p\,r^{N-1}\,t^{N-1}\, \frac{r^2 + t^2}{|r^2-t^2|^{N+1}}\,dr\,dt \\ & = \frac{2}{|\ln{\varepsilon}|}\int_\varepsilon^1 |\ln{r}-\ln{\varepsilon}|^p\, r^{N-1}\left(\int_0^\varepsilon t^{N-1}\, \frac{r^2 + t^2}{|r^2-t^2|^{N+1}}\,dt\right)\,dr \\ & = \frac{2\varepsilon^N}{N|\ln{\varepsilon}|}\int_\varepsilon^1 \frac{|\ln{r}-\ln{\varepsilon}|^p}{(r-\varepsilon)^N}\, \frac{r^{N-1}}{(r+\varepsilon)^N}\,dr \\ & = \frac{2}{N|\ln{\varepsilon}|}\int_1^{\frac{1}{\varepsilon}} \frac{|\ln{x}|^p}{(x-1)^N}\, \frac{x^{N-1}}{(1+x)^N}\,dx
\end{align*}
where we applied the change of variable $x=\frac{r}{\varepsilon}$. Since $p>N$, the integral
\[ \int_1^{+\infty} \frac{|\ln{x}|^p}{(x-1)^N}\, \frac{x^{N-1}}{(1+x)^N}\,dx \]
is convergent and therefore $\lim_{\varepsilon \to 0} I_1(\varepsilon)=0$.

\subsection{Computations for $I_2(\varepsilon)$}

\begin{align*}
 & \frac{1}{|\ln{\varepsilon}|}\int_\varepsilon^1 \int_\varepsilon^1 |\ln{r}-\ln{t}|^p\,r^{N-1}\,t^{N-1}\, \frac{r^2 + t^2}{|r^2-t^2|^{N+1}}\,dr\,dt \\ & = \frac{1}{|\ln{\varepsilon}|}\int_\varepsilon^1 \frac{1}{t^2 } \left(\int_\varepsilon^1 \bigg|\ln{\frac{r}{t}}\bigg|^p\,\left(\frac{r}{t}\right)^{N-1}\, \frac{\left(\frac{r}{t}\right)^2 + 1}{|\left(\frac{r}{t}\right)^2-1|^{N+1}}\,dr\right)\,dt \\ & = \frac{1}{|\ln{\varepsilon}|}\int_\varepsilon^1 \frac{1}{t } \left(\int_{\frac{\varepsilon}{t}}^{\frac{1}{t}} |\ln{x}|^p\,x^{N-1}\, \frac{x^2 + 1}{|x^2-1|^{N+1}}\,dx\right)\,dt.
\end{align*}
After performing an integration by parts, the last quantity is equal to
\begin{align*}
 & \frac{1}{|\ln{\varepsilon}|}\int_\varepsilon^1 \frac{\ln{t}}{t^2} \bigg|\ln{\frac{1}{t}}\bigg|^p\,\frac{1}{t^{N-1}}\, \frac{\frac{1}{t^2} + 1}{|\frac{1}{t^2}-1|^{N+1}}\,dt-\frac{\varepsilon}{|\ln{\varepsilon}|}\int_\varepsilon^1 \frac{\ln{t}}{t^2}\,\bigg|\ln{\frac{\varepsilon}{t}}\bigg|^p\,\left(\frac{\varepsilon}{t}\right)^{N-1}\, \frac{\left(\frac{\varepsilon}{t}\right)^2 + 1}{|\left(\frac{\varepsilon}{t}\right)^2-1|^{N+1}}\,dt \\ & + \frac{1}{|\ln{\varepsilon}|}\left[ \ln{t} \left(\int_{\frac{\varepsilon}{t}}^{\frac{1}{t}} |\ln{x}|^p\,x^{N-1}\, \frac{x^2 + 1}{|x^2-1|^{N+1}}\,dx\right)\right]^{t=1}_{t=\varepsilon}.
\end{align*}
For $\varepsilon \to 0$, the first term converges to $0$, since the integral
\[ \int_0^1 \frac{\ln{t}}{t^2} \bigg|\ln{\frac{1}{t}}\bigg|^p\,\frac{1}{t^{N-1}}\, \frac{\frac{1}{t^2} + 1}{|\frac{1}{t^2}-1|^{N+1}}\,dt\]
is convergent. The second term is equal, after a change of variable, to
\begin{align*} & -\frac{1}{|\ln{\varepsilon}|}\int_\varepsilon^1 \ln{\left(\frac{\varepsilon}{x} \right)}\,|\ln{x}|^p\,x^{N-1}\, \frac{x^2 + 1}{|x^2-1|^{N+1}}\,dx \\ & = \int_\varepsilon^1 |\ln{x}|^p\,x^{N-1}\, \frac{x^2 + 1}{|x^2-1|^{N+1}}\,dx - \frac{1}{|\ln{\varepsilon}|}\int_\varepsilon^1 |\ln{x}|^{p+1}\,x^{N-1}\, \frac{x^2 + 1}{|x^2-1|^{N+1}}\,dx \end{align*}
and it converges to
\[ \int_0^1 |\ln{x}|^p\,x^{N-1}\, \frac{x^2 + 1}{|x^2-1|^{N+1}}\,dx = \int_1^{+\infty} |\ln{x}|^p\,x^{N-1}\, \frac{x^2 + 1}{|x^2-1|^{N+1}}\,dx\]
as $\varepsilon \to 0$. The third term is equal to
\[ \int_{1}^{\frac{1}{\varepsilon}} |\ln{x}|^p\,x^{N-1}\, \frac{x^2 + 1}{|x^2-1|^{N+1}}\,dx \]
so that finally
\[ \lim_{\varepsilon \to 0} I_2(\varepsilon) = 2\int_{1}^{+\infty} |\ln{x}|^p\,x^{N-1}\, \frac{x^2 + 1}{|x^2-1|^{N+1}}\,dx.\]

\subsection{Computations for $I_3(\varepsilon)$}
\begin{align*}
 & 2|\ln{\varepsilon}|^{p-1}\int_1^{+\infty} \int_0^\varepsilon r^{N-1}\,t^{N-1}\, \frac{r^2 + t^2}{|r^2-t^2|^{N+1}}\,dr\,dt \\ & = 2|\ln{\varepsilon}|^{p-1}\int_1^{+\infty}  r^{N-1}\, \left(\int_0^\varepsilon  t^{N-1}\, \frac{r^2 + t^2}{|r^2-t^2|^{N+1}}\,dt\right)\,dr \\ & = \frac{2\varepsilon^N|\ln{\varepsilon}|^{p-1}}{N}\int_1^{+\infty} \frac{r^{N-1}}{(r^2-\varepsilon^2)^N}\,dr
\end{align*}
and therefore $\lim_{\varepsilon \to 0} I_3(\varepsilon) \to 0$.

\subsection{Computations for $I_4(\varepsilon)$}
\begin{align*}
 & \frac{2}{|\ln{\varepsilon}|}\int_\varepsilon^1 \int_1^{+\infty} |\ln{r}|^p\,r^{N-1}\,t^{N-1}\, \frac{r^2 + t^2}{|r^2-t^2|^{N+1}}\,dr\,dt \\ & = \frac{2}{|\ln{\varepsilon}|}\int_\varepsilon^1 |\ln{r}|^p\,r^{N-1} \left( \int_1^{+\infty} t^{N-1}\, \frac{r^2 + t^2}{|r^2-t^2|^{N+1}}\,dt\right)\,dr \\ & = \frac{2}{N|\ln{\varepsilon}|}\int_\varepsilon^1 |\ln{r}|^p\,\frac{r^{N-1}}{(1-r^2)^N}\,dr.
\end{align*}
Therefore $\lim_{\varepsilon \to 0} I_4(\varepsilon) = 0$, since the integral
\[ \int_0^1 |\ln{r}|^p\,\frac{r^{N-1}}{(1-r^2)^N}\,dr \]
is finite (the integrand function is bounded since $p>N$).

\subsection{Computation of the integral} \label{computationintegral}

The value of the integral
\[ \int_{1}^{+\infty} |\ln{x}|^p\,x^{N-1}\, \frac{x^2 + 1}{(x^2-1)^{N+1}}\,dx \]
can be computed explicitly. To this aim, we write
\[ \int_{1}^{+\infty} |\ln{x}|^p\,x^{N-1}\, \frac{x^2 + 1}{(x^2-1)^{N+1}}\,dx = \frac{p}{N}\int_{1}^{+\infty} |\ln{x}|^{p-1}\,\frac{x^{N-1}}{(x^2-1)^{N}}\,dx = \frac{p}{N}\int_{0}^{1} |\ln{t}|^{p-1}\,\frac{t^{N-1}}{(1-t^2)^{N}}\,dt. \]
For $p \in \mathbb{N}$, this can be found for instance in \cite[2.6.5.1, p. 490]{prudnikov}. Otherwise, for a generic $p \in \real^+$, we write
\[ \frac{1}{(1-x^2)^N} = \sum_{k=0}^\infty \frac{(N+k-1)!}{k!\,(N-1)!}x^{2k}  \]
(see \cite[5.2.2.10, p. 697]{prudnikov}). According to \cite[2.6.3.1, p. 488]{prudnikov}
\begin{align*} \int_{0}^{1} |\ln{t}|^{p-1}\,\frac{t^{N-1}}{(1-t^2)^{N}}\,dt & = \int_{0}^{1} |\ln{t}|^{p-1}\,\sum_{k=0}^\infty \frac{(N+k-1)!}{k!\,(N-1)!}t^{N+2k-1}\,dt \\ & = \sum_{k=0}^\infty \frac{(N+k-1)!}{k!\,(N-1)!}\left(\int_{0}^{1} |\ln{t}|^{p-1}\, t^{N+2k-1}\,dt\right) \\ & = \frac{\Gamma(p)}{(N-1)!} \sum_{k=0}^\infty \frac{(N+k-1)!}{k!}\frac{1}{(N+2k)^p}.\end{align*}
Let us compute the last quantity for some values of $N$.
For $N=2$, we have
\[ \Gamma(p)\sum_{k=0}^\infty \frac{(k+1)!}{k!}\frac{1}{(2+2k)^p} = \frac{\Gamma(p)}{2^p} \sum_{k=0}^\infty \frac{1}{(k+1)^{p-1}} = \frac{\Gamma(p)}{2^p} \sum_{k=1}^\infty \frac{1}{k^{p-1}} = \frac{\Gamma(p)}{2^p}\zeta(p-1).  \]
For $N=3$, we have
\begin{align*} \frac{\Gamma(p)}{2}\sum_{k=0}^\infty \frac{(k+2)!}{k!}\frac{1}{(3+2k)^p} & = \frac{\Gamma(p)}{2^{p+1}} \sum_{k=0}^\infty \frac{k^2+k}{\left(k+\frac{1}{2}\right)^{p}} = \frac{\Gamma(p)}{2^{p+1}} \sum_{k=0}^\infty \left[\frac{1}{\left(k+\frac{1}{2}\right)^{p-2}} - \frac{1}{4} \frac{1}{\left(k+\frac{1}{2}\right)^{p}}\right] \\ & = \frac{\Gamma(p)}{2^{p+1}}\left[\zeta\left(p-2,\frac{1}{2}\right)-\frac{1}{4}\zeta\left(p,\frac{1}{2}\right)\right]\\ &  = \frac{\Gamma(p)}{2^{p+1}}\left[(2^{p-2}-1)\zeta(p-2) - 2^{-2}(2^p-1)\zeta(p)\right] \end{align*}
where $\zeta(s,q)$ is Hurwitz zeta function.
For $N=4$, we have
\begin{align*} \frac{\Gamma(p)}{6}\sum_{k=0}^\infty \frac{(k+3)!}{k!}\frac{1}{(4+2k)^p} & = \frac{\Gamma(p)}{3 \cdot 2^{p+1}} \sum_{k=0}^\infty \frac{k^2+4k+3}{(k+2)^{p-1}} \\ & = \frac{\Gamma(p)}{3 \cdot 2^{p+1}} \sum_{k=0}^\infty \left[\frac{1}{(k+2)^{p-3}} -  \frac{1}{(k+2)^{p-1}}\right] \\ & = \frac{\Gamma(p)}{3 \cdot 2^{p+1}}[\zeta(p-3)-\zeta(p-1)]. \end{align*}

In general, for $N=2m$ even, we can write
\[ \frac{(N+k-1)!}{k!} = \prod_{j=1}^{N-1} (k+j) = (k+m)^{N-1} + \sum_{i=2}^{N-2} a_i(k+m)^{N-i}\]
so that
\begin{align*} \sum_{k=0}^\infty \frac{(N+k-1)!}{k!}\frac{1}{(N+2k)^p} & = \frac{1}{2^p}\sum_{k=0}^\infty \frac{(N+k-1)!}{k!}\frac{1}{(k+m)^p} \\ & = \frac{1}{2^p}\sum_{k=0}^\infty \left[\frac{1}{(k+m)^{p+1-N}} + \sum_{i=2}^{N-1} \frac{a_i}{(k+m)^{p+i-N}}\right] \\ & = \frac{1}{2^p} \left[\zeta(p+1-N) + \sum_{i=2}^{N-1} a_i\zeta(p+i-N) + c \right]\end{align*}
and thus
\[ \lim_{p \to N} (p-N)\sum_{k=0}^\infty \frac{(N+k-1)!}{k!}\frac{1}{(N+2k)^p} = \frac{1}{2^N}. \]
For $N=2m+1$ odd,
we can write
\[ \frac{(N+k-1)!}{k!} = \prod_{j=1}^{N-1} (k+j) = \left(k+\frac{N}{2}\right)^{N-1} + \sum_{i=2}^{N-2} a_i\left(k+\frac{N}{2}\right)^{N-i}\]
so that
\begin{align*} \sum_{k=0}^\infty \frac{(N+k-1)!}{k!}\frac{1}{(N+2k)^p} & = \frac{1}{2^p}\sum_{k=0}^\infty \frac{(N+k-1)!}{k!}\frac{1}{\left(k+\frac{N}{2}\right)^p} \\ & = \frac{1}{2^p}\sum_{k=0}^\infty \left[\frac{1}{\left(k+\frac{N}{2}\right)^{p+1-N}} + \sum_{i=2}^{N-1} \frac{a_i}{\left(k+\frac{N}{2}\right)^{p+i-N}}\right] \\ & = \frac{1}{2^p} \left[\zeta\left(p+1-N,\frac{1}{2}\right) + \sum_{i=2}^{N-1} a_i\zeta\left(p+i-N,\frac{1}{2}\right) + c \right]\end{align*}
and again, the leading term is $\zeta\left(p+1-N,\frac{1}{2}\right) = (2^{p+1-N}-1)\zeta(p+1-N)$, so that
\[ \lim_{p \to N} (p-N)\sum_{k=0}^\infty \frac{(N+k-1)!}{k!}\frac{1}{(N+2k)^p} = \frac{1}{2^N}. \]
We can summarize the results we found in the following proposition.

\begin{prop}
Let $\{u_\varepsilon\}$ be the family of functions in ${\widetilde W}^{s,p}_0(B)$ ($sp=N$) defined in \eqref{maxsequence}. Then,
\[ \lim_{\varepsilon \to 0} [u_\varepsilon]_{W^{s,p}(\real^N)}^p = \gamma_{s,N} := \frac{2\,(N\omega_N)^2\,\Gamma(p+1)}{N!}\,\sum_{k=0}^\infty \frac{(N+k-1)!}{k!}\frac{1}{(N+2k)^p}.\]
Moreover,
\[ \lim_{s \to 1^-} (1-s)\gamma_{s,N} = \lim_{p \to N} \frac{p-N}{p}\gamma_{s,N} = \frac{N\omega_N^2}{2^{N-1}}.\]
\end{prop}

We are now ready to prove the main result of this section.

\begin{prop}
Let $\Omega$ be a bounded, open domain of $\real^N$, and let $sp = N$. There exists $\alpha^*_{s,N} := N(\gamma_{s,N})^\frac{s}{N-s}$ such that
\[ \sup \left\{\int_\Omega \exp{\left(\alpha\,|u|^{\frac{N}{N-s}}\right)}\,\bigg|\,u \in \widetilde{W}^{s,p}_0(\Omega),\,[u]_{W^{s,p}(\real^N)}\leq 1 \right\} = + \infty \qquad \text{for } \alpha \in (\alpha^*_{s,N}, +\infty).\]
\end{prop}
\begin{proof}
As we discussed before, it is enough to consider the case $\Omega=B$. Let $u_\varepsilon$ be the concentrating family defined in \eqref{maxsequence}, which satisfies $[u_\varepsilon]_{W^{s,p}(\real^N)} \to (\gamma_{s,N})^{\frac{1}{p}}$ as $\varepsilon \to 0$. Let $\alpha > N(\gamma_{s,N})^\frac{s}{N-s}$. For $\varepsilon$ near to $0$, we will have $\alpha\,[u_\varepsilon]_{W^{s,p}(\real^N)}^{-\frac{N}{N-s}} \geq \beta > N$. Set $v_\varepsilon := \frac{1}{[u_\varepsilon]_{W^{s,p}(\real^N)}}u_\varepsilon$. Then
\begin{align*} \int_{B_1} \exp{\left(\alpha |v_\varepsilon|^{\frac{N}{N-s}}\right)} \geq \int_{B_\varepsilon} \exp{\left(\alpha |v_\varepsilon|^{\frac{N}{N-s}}\right)} \geq \int_{B_\varepsilon} \exp{(-\beta\ln{\varepsilon})}  = \omega_N\,\varepsilon^{N-\beta} \to + \infty\end{align*}
as $\varepsilon \to 0$, therefore proving the claim.
\end{proof}

\begin{rem}
It is also interesting to observe that, for $\alpha < N(\gamma_{s,N})^\frac{s}{N-s}$, the quantity
\[ \int_{B} \exp{\left(\alpha |v_\varepsilon|^{\frac{N}{N-s}}\right)} \leq \int_{B}  \exp{\left(\beta |u_\varepsilon|^{\frac{N}{N-s}}\right)} \]
where $\alpha\,[u_\varepsilon]_{W^{s,p}(\real^N)}^{-\frac{N}{N-s}} \leq \beta < N$, is uniformly bounded. This amounts to prove that the limit
\[ \lim_{\varepsilon \to 0} \int_\varepsilon^1 \exp{\left(\beta \frac{|\ln{t}|^{\frac{N}{N-s}}}{|\ln{\varepsilon}|^{\frac{s}{N-s}}}\right)\,t^{N-1}\,dt} \]
is finite. In fact we have
\[ \exp{\left(\beta \frac{|\ln{t}|^{\frac{N}{N-s}}}{|\ln{\varepsilon}|^{\frac{s}{N-s}}}\right)}\,t^{N-1} \leq \exp{\left(\beta \frac{|\ln{t}|^{\frac{N}{N-s}}}{|\ln{t}|^{\frac{s}{N-s}}}\right)}\,t^{N-1} \leq t^{N-1-\beta}\]
and
\[ \int_\varepsilon^1 t^{N-1-\beta}\,dt = \frac{1}{N-\beta}\left(1-\varepsilon^{N-\beta}\right) \to \frac{1}{N-\beta}\]
as $\varepsilon \to 0$.
\end{rem}

\begin{rem}
We observe that the results obtained are consistent with the local case. For example, if $N=2$, thanks to the computations of Section \ref{computationintegral} we have
\[ \lim_{s \to 1^-} (1-s)\alpha^*_{s,2} = 2\pi^2\]
which coincides with the optimal exponent $\alpha^*_{1,2}=4\pi$ (see \cite{moser}), up to the multiplicative constant
\[ K(2,2):= \frac{1}{2}\int_{S^1} |\langle \sigma,\mathbf{e}\rangle|^2\,d\mathcal{H}^{N-1}(\sigma) = \frac{\pi}{2} \]
which appears in the asymptotic behaviour of Gagliardo seminorms in the limit $s \to 1^-$ (see \eqref{essea1}).

%and denote by $v_{s,\rho}$ the corresponding functions in \eqref{maxsequence}. They satisfy
% \[ [v_{s,\rho}]_{W^{s,p}(\real^N)}^p \to \frac{p}{2^{p-2}}\pi^2 \Gamma(p) \zeta(p-1)\]
% as $\rho \to 0$, while for $p=N=2$,
% \[ \|\nabla v_{1,\rho}\|_2^2 = 2\pi. \]
% As $s \to 1^-$ (namely $p \to 2$), we have
% \[ (1-s)\, \frac{p}{2^{p-2}}\pi^2 \Gamma(p) \zeta(p-1) = \left(\frac{p-2}{p} \right)\frac{p}{2^{p-2}}\pi^2 \Gamma(p)\zeta(p-1) \to \pi^2. \]
% In other words,
% \[ \lim_{s \to 1^-} \lim_{\rho \to 0} (1-s)[v_{s,\rho}]_{W^{s,N/s}(\real^N)}^p =  \lim_{\rho \to 0} \lim_{s \to 1^-} K(2,2)\,\|\nabla v_{s,\rho}\|_2^2,\]
% where
% \[ K(2,2) = \frac{1}{2}\int_{S^1} |\langle \sigma,\mathbf{e}\rangle|^2\,d\mathcal{H}^{N-1}(\sigma) = \frac{\pi}{2}.\]
% This is coherent with the result stated in \eqref{essea1}, even if that result holds true for \emph{fixed} $p$ and varying $s$. One may wonder if, for a family of functions $\{u_s\}$ in $W^{1,N}_0(\Omega)$ converging in $L^N(\Omega)$ to a function $u \in W^{1,N}_0(\Omega)$, it always holds
% \[ \lim_{s \to 1^-} (1-s)[u_s]_{W^{s,N/s}(\real^N)}^{\frac{N}{s}} = K(N,N)\|\nabla u\|_N^N,\]
% giving an extension of \cite[Theorem 3.1]{brascoparinisquassina} which states that
% \[ \lim_{s \to 1^-} (1-s)[u_s]_{W^{s,N}(\real^N)}^{N} = K(N,N)\|\nabla u\|_N^N.\]
% (observe that $W^{1,N}_0(\Omega) \hookrightarrow {\widetilde W}^{s,N/s}_0(\Omega)$ according to \cite[Theorem 1.4.4.1]{grisvard}).
\end{rem}

\section{Conclusions and final remarks}

The present work is a first step towards the understanding of the Moser-Trudinger inequality in fractional Sobolev-Slobodeckij spaces. However, we are left with many open questions. Let $\alpha^*_{opt}$ be the supremum of all $\alpha > 0$ such that the Moser-Trudinger inequality holds true. The first question concerns of course the optimality of the exponent $\alpha^*_{s,N}$ of Section \ref{upper}: does it hold $\alpha^*_{s,N}=\alpha^*_{opt}$? Moreover, does the Moser-Trudinger inequality hold true also for $\alpha=\alpha^*_{opt}$, as in the classical case? And if this is the case, is the supremum attained, similarly to the results of \cite{carlesonchang} and \cite{lin}?

Once the optimal exponent is determined, it would be interesting to compare it with its counterpart for the Bessel potential spaces $\widetilde{H}^{s,p}(\Omega)$ defined in \cite{martinazzi}. Since $\widetilde{H}^{s,2}(\Omega) = \widetilde{W}^{s,2}_0(\Omega)$, the reader might wonder whether $\alpha^*_{s,N}$ coincides with the optimal exponent for $\widetilde{H}^{s,2}(\Omega)$. Although we restricted the analysis to the case $N \geq 2$, it is interesting to observe that, for $s=\frac{1}{2}$ and $N=1$, our exponent $\alpha^*_{\frac{1}{2},1}$ is equal to $2\pi^2$ and, up to a normalization constant due to the fact that
\[ [u]_{W^{\frac{1}{2},2}(\real)}^2=2\pi [u]_{H^{\frac{1}{2},2}(\real)}^2 \]
(see \cite[Proposition 3.6]{dinezzapalatuccivaldinoci}), it coincides with the optimal exponent $\alpha_2=\pi$ determined in \cite{iulamaalaouimartinazzi} for the space $\widetilde{H}^{\frac{1}{2},2}(I)$. 
Another possible research direction could be the investigation of whether a generalized inequality of the type
\[ \sup \left\{\int_\Omega f(|u|)\,\exp{\left(\alpha\,|u|^{\frac{N}{N-s}}\right)}\,\bigg|\,u \in \widetilde{W}^{s,p}_0(\Omega),\,[u]_{W^{s,p}(\real^N)}\leq 1 \right\}< + \infty, \]
where $f:\real^+\to \real^+$ is such that $f(t)\to +\infty$ as $t \to +\infty$, holds true for the same exponents of the standard Moser-Trudinger inequality. In the case of Bessel potential spaces, this was investigated in \cite{iulamaalaouimartinazzi} for $N=1$, and in \cite{hyder} for $N \geq 2$.

\end{document}